\documentclass[12pt,a4paper]{article}
\usepackage{amsmath,amssymb,amsthm}
\usepackage{amssymb,latexsym, bbm,comment}
\usepackage{graphicx}
\usepackage{makeidx}
\newtheorem{theorem}{Theorem}[section]
\newtheorem{lemma}[theorem]{Lemma}
\newtheorem{prop}[theorem]{Proposition}
\newtheorem{cor}[theorem]{Corollary}
\numberwithin{equation}{section}

\newcommand{\R}{\mathbb{R}}
\newcommand{\Rd}{\mathbb{R}^d}

\newcommand{\C}{\mathbb{C}}

\newcommand{\bean}{\begin{eqnarray*}}
\newcommand{\eean}{\end{eqnarray*}}

\newcommand{\Z}{\mathbb{Z}}

\begin{document}

\date{}

\title{On the Spectrum of Self--Adjoint L\'{e}vy Generators}

\author{David Applebaum\\ School of Mathematics and Statistics,\\ University of
Sheffield,\\ Hicks Building, Hounsfield Road,\\ Sheffield,
England, S3 7RH\\ ~~~~~~~\\e-mail: D.Applebaum@sheffield.ac.uk}

\maketitle

\begin{abstract} We investigate the spectrum of the generator of a self-adjoint transition semigroup of a (symmetric) L\'{e}vy process taking values in $d$--dimensional space.
\end{abstract}

\section{Introduction} L\'{e}vy processes are essentially stochastic processes with independent increments. They have been extensively studied in recent years, as they include not only the most important Gaussian process -- Brownian motion, but also many rich processes with jumps such as the $\alpha$-stable ones. Four monograph treatments have been published in the last twenty five years that explore different aspects of their theory and application \cite{Be,Sa,Appbk,Kyp}. The transition semigroups induced by the process may be studied in various Banach spaces, such as the space of continuous functions that vanish at infinity and the $L^{p}$--spaces. The fact that these operators, and their generators, may be written as pseudo--differential operators, whose symbol is given by the characteristic exponent of the process, has led to extensive work on L\'{e}vy--type operators that generate a very large class of Feller processes -- see e.g. \cite{Jac1,Jac2,Jac3,Jac4,BSW}. In this short note, we are interested in the generator as an operator acting in $L^{2}(\Rd)$.

For a long time, necessary and sufficient conditions have been known for when such an operator is self--adjoint and this material has recently been reviewed in Chapter 4 of \cite{Appnew}. Such operators are of interest in mathematical physics as they include both the free non--relativistic Hamiltonian (i.e. the (negative) Laplacian), and its relativistic counterpart. Some deep perturbation theory studies of these operators (at the form sum level), have been carried out in \cite{HS}, while \cite{CMS} has studied the relationship between recurrence of the underlying L\'{e}vy process and the existence of at least one negative bound state under perturbation by  negative bounded potentials with compact support. However the author is not aware of any studies of the spectrum of the operator in question. For the case of L\'{e}vy processes on the torus, the spectrum consists of eigenvalues corresponding to the range of the characteristic exponent (see section 5.2.3 in \cite{Appnew}). When we work on $\Rd$, we expect to obtain a continuous spectrum which is again given by the range of the characteristic exponent. After collecting together all the known facts we need in section 2, we will prove this conjecture in section 3 and present some examples. We use only elementary techniques in this note, the main purpose of which is to stimulate interest in the somewhat neglected area where spectral theory meets probability.

\section{Preliminaries}

The material in the section can be found in Chapters 3 and 4 of \cite{Appnew}. Much of it is also in Chapter 3 of \cite{Appbk}. Let $(\mu_{t}, t \geq 0)$ be a weakly continuous convolution semigroup of probability measures in $\Rd$. These arise naturally as the collections of laws of a L\'{e}vy process taking values in $\Rd$.
We obtain a strongly continuous semigroup  $(T_{t}, t \geq 0)$ of contractions in $L^{2}(\Rd)$ by the prescription

$$ T_{t}f(x) = \int_{\Rd}f(x + y)\mu_{t}(dy),$$
for all $x \in \Rd, t \geq 0, f \in L^{2}(\Rd)$.

In this article, we will require that $\mu_{t}$ is symmetric for all $t \geq 0$, i.e. $\mu_{t}(A) = \mu_{t}(-A)$, for all Borel sets $A$ in $\Rd$. In this case the semigroup $(T_{t}, t \geq 0)$ is self--adjoint, and its self--adjoint infinitesimal generator $A$ acts on the space $C^{2}_{c}(\Rd)$ of twice continuously differentiable functions with compact support on $\Rd$ by the prescription

\begin{eqnarray} \label{gen}
Af(x) & = & \sum_{i,j = 1}^{d}a_{ij} \partial_{i} \partial_{j}f(x) \nonumber \\
& + & \int_{\Rd}(f(x + y) - 2f(x) + f(x - y))\nu(dy),
\end{eqnarray}

for all $f \in C^{2}_{c}(\Rd), x \in \Rd$, where $a = (a_{ij})$ is a non--negative definite symmetric $d \times d$ matrix, and $\nu$ is a symmetric L\'{e}vy measure, i.e. $\nu(\{0\}) = 0$ and $\int_{\Rd}(|y|^{2} \wedge 1) \nu(dy) < \infty$ (see Chapter 4 of \cite{Appnew} for the proof).

We have the L\'{e}vy--Khintchine formula

$$ \int_{\Rd}e^{i u \cdot x}\mu_{t}(dy) = e^{-t \eta(u)},$$

where $\eta$ is a continuous, negative definite function with $\eta(0) = 0$ that takes the form

\begin{equation} \label{symbol}
\eta(u) = au \cdot u + \int_{\Rd}(1 - \cos(u \cdot y))\nu(dy),
\end{equation}
for all $u \in \Rd$.

We define the usual Fourier transform $\widehat{f}$ of $f \in L^{2}(\Rd)$ by

$$ \widehat{f}(y) = \frac{1}{(2\pi)^{d/2}}\int_{\Rd}e^{-i x \cdot y}f(x)dx,$$
for $y \in \Rd$.

The domain of the generator $A$ is the non--isotropic Sobolev space
$$ {\mathcal H}_{\eta}(\Rd) = \{f \in L^{2}(\Rd), \int_{\Rd}\eta(y)^{2}|\widehat{f}(y)|^{2}dy < \infty\}.$$
Clearly $C^{2}_{c}(\Rd) \subseteq {\mathcal H}_{\eta}(\Rd)$, and on this larger domain we have the pseudo--differential operator representation

\begin{equation} \label{pseudo}
Af(x) = -\frac{1}{(2\pi)^{d/2}}\int_{\Rd}e^{i x \cdot y}\eta(y)\widehat{f}(y)dy.
\end{equation}

Underlying (\ref{pseudo}) is the formal relation

$$ Ae^{i u \cdot x} = -\eta(u)e^{i u \cdot x},$$

which suggests that the spectrum $\sigma(A) = \mbox{Ran}(-\eta)$. This is certainly true when $A$ is the Laplacian ($a_{ij} = \delta_{ij}$ and $\nu = 0$). We investigate it next in the more general context.

\section{Results}

Let $B$ be a closed linear operator in a real Banach space $E$ with domain $D_{B} \subseteq E$. We recall that its spectrum $\sigma(B)$ is the closed subset of $\C$ defined by
$$ \sigma(B):=\{\lambda \in \C; \lambda I - B~\textrm{fails to be invertible from}~D_{B}~\mbox{to}~E\}.$$
The resolvent set of $B$ is $\rho(B): = \C \setminus \sigma(B)$. We need a general result. It is surely well--known, but we include a proof for the reader's convenience.

\begin{prop} \label{genres}
If $(S_{t}, t \geq 0)$ is a strongly continuous self--adjoint contraction semigroup acting in a Hilbert space $H$ and having infinitesimal generator $B$ then
$$ \sigma(B) \subseteq (-\infty, 0].$$
\end{prop}

\begin{proof} Since $B$ generates a contraction semigroup, by the Hille--Yosida theorem, $\sigma(B) \subseteq \{z \in \C; \Re(z) \leq 0\}$, but $B$ is self--adjoint and so $\sigma(B) \subseteq \R$. The result follows.
\end{proof}

Our approach to investigating the spectrum of $A$ is by a standard argument using Fourier transforms, as is well--known for the case of the Laplacian.

\begin{theorem} \label{incl}
$$ \sigma(A) = \mbox{Ran}(-\eta).$$
\end{theorem}

\begin{proof}
Consider the equation
$$ (\lambda I - A)f = g,$$
for $\lambda \in \C, f \in {\mathcal H}_{\eta}(\Rd), g \in L^{2}(\Rd)$. Taking Fourier transforms of both sides we obtain
$$ (\lambda + \eta(y))\widehat{f}(y) = \widehat{g}(y),$$
for all $y \in \Rd$. If the operator equation has a solution, then
$$(1 + \eta(\cdot))\widehat{f}(\cdot) = \frac{1 + \eta(\cdot)}{\lambda + \eta(\cdot)}\widehat{g}(\cdot) \in L^{2}(\Rd).$$
This clearly cannot be the case when $\lambda = - \eta(y)$ for some $y \in \Rd$, and it follows that $\mbox{Ran}(-\eta) \subseteq \sigma(A)$.

Our goal now is to prove that the resolvent set $\rho(A) = \C \setminus \mbox{Ran}(-\eta)$. By Proposition \ref{genres}, we know that $\C \setminus (-\infty, 0] \subseteq \rho(A)$. So it is sufficient to consider $\lambda \in (-\infty, 0) \setminus \mbox{Ran}(-\eta)$. The required result holds if $g_{\lambda}$ is bounded, where for all $u \in \Rd$,
$$ g_{\lambda}(u): = \frac{1 + \eta(\cdot)}{\lambda + \eta(\cdot)}.$$

There are two cases to consider. Firstly assume that $\sup_{u \in \Rd}\eta(u) = K < \infty$. Then $\mbox{Ran}(\eta) = [0, K]$. Choose $\epsilon > 0$ and assume that $\lambda \in (\-\infty, -K -\epsilon)$.
Then we find that
$$ \sup_{u \in \Rd}|g_{\lambda}(u)| \leq \frac{1 + K}{\epsilon} < \infty.$$ But then we conclude that
$$ (-\infty, -K) = \bigcup_{\epsilon > 0}(-\infty, -K -\epsilon) \subseteq \rho(A),$$ and so $\sigma(A) = \mbox{Ran}(-\eta)$, as required.
If $\sup_{u \in \Rd}\eta(u) = \infty$, then by the first part of the proof we have $(-\infty, 0) \subseteq \sigma(A)$ and so the required result follows directly from Proposition \ref{genres}.


\end{proof}

\begin{cor} \label{lim}
If $\lim_{u \rightarrow \infty}\eta(u) = \infty$, then $\sigma(A) = (- \infty, 0]$.
\end{cor}

\begin{proof} This is an immediate consequence of Theorem \ref{incl} and the continuity of $\eta$.
\end{proof}

We note that the condition of Corollary \ref{lim} was also required in \cite{CMS} to obtain the result alluded to in the introduction.

Some well--known and important examples satisfy the conditions of Corollary \ref{lim}. These include
\begin{itemize}
\item The Laplacian $\Delta$, where $\eta(u) = |u|^{2}$.

\item The fractional Laplacian $-(-\Delta)^{\alpha/2}$, where $\eta(u) = |u|^{\alpha}$ with $0 < \alpha < 2$.

\item The relativistic Schr\"{o}dinger operator $-\sqrt{b^{2}I - \Delta} + bI$, where $\eta(u) = \sqrt{b^{2} + |u|^{2}} - b$ for some $b > 0$.

\end{itemize}

More generally, to obtain $\sigma(A) = (- \infty, 0]$ we may take $A = -f(-\Delta)$, where $f$ is a strictly increasing Bernstein function for which $f(0) = 0$, in which case the underlying L\'{e}vy process is a subordinated Brownian motion (see e.g. \cite{SSV}). The above three examples are all special cases of this class.

To find other interesting class of examples, we first write $A = A_{1} + A_{2}$, where $A_{1}$ is the second order differential operator in (\ref{gen}).

\begin{lemma} \label{bound} If $\nu$ is a finite measure, then $A_{2}$ is a bounded operator in $L^{2}(\Rd)$, with $\sigma(A_{2}) \subseteq [-2M, 0]$.
\end{lemma}

\begin{proof} Let $M := \nu(\Rd)$ and define $A_{3}f(x) = \int_{\Rd}(f(x+y) - f(x))\nu(dy)$, for $f \in L^{2}(\Rd), x \in \Rd$. The linear operator $A_{3}$ is bounded as by the Cauchy--Schwarz inequality and Fubini's theorem,
\bean ||A_{3}f||^{2} & = & \int_{\Rd}\left(\int_{\Rd}(f(x+y) - f(x))\nu(dy)\right)^{2}dx \\
& \leq & M \int_{\Rd} \int_{\Rd}(f(x+y) - f(x))^{2}dx \nu(dy) \\
& \leq & 2M \int_{\Rd} \int_{\Rd}f(x+y)^{2}dx \nu(dy) + 2M^{2}||f||^{2} \\
& \leq & 4M^{2}||f||^{2}, \eean
and the first result follows by observing that $A_{2} = 2A_{3}$.
The second result follows from the fact that for all $u \in \Rd$
$$\eta(u) = \int_{\Rd}(1 - \cos(u \cdot y))\nu(dy) \leq 2M.$$
\end{proof}

A necessary and sufficient condition for $\sigma(A_{2}) = [-2M, 0]$ is that
$$ \nu(\{y \in \Rd; u\cdot y = (2n+1)\pi~\mbox{for some}~u \in \Rd, n \in \Z\}) > 0.$$ Now let $\alpha > 0$ be arbitrary. For an example where $\sigma(A_{2}) = [-\alpha, 0]$, take $d = 1$ and $\nu = \frac{\alpha}{4}(\delta_{-1} + \delta_{1})$.

We conjecture that $A$ is bounded if and only if $a = 0$ and $\nu$ is finite.

\begin{cor} If $\nu$ is finite and there exist $b_{i} > 0$ for all $i=1, \ldots, d$ such that $$au \cdot u \geq \sum_{i=1}^{d}b_{i}u_{i}^{2}$$ for all $u \in \Rd$, then $\sigma(A) = (-\infty, 0)$.
\end{cor}

\begin{proof} This follows from the fact that
$\eta(u) \geq \sum_{i=1}^{d}b_{i}u_{i}^{2} \rightarrow \infty$, as $u \rightarrow \infty.$
\end{proof}

Note that the condition in the last corollary is satisfied if $a$ is bounded below in the sense that there exists $b > 0$ so that $a \geq bI$.

\end{document}